\documentclass[11pt]{article} 
\usepackage{etex}
\usepackage[utf8]{inputenc} 

\usepackage{amsthm,amssymb,amsmath}
\usepackage{mathtools}

\usepackage{geometry} 
\usepackage{graphicx} 
\usepackage{color}

\usepackage{booktabs} 
\usepackage{array} 
\usepackage{paralist} 
\usepackage{verbatim} 
\usepackage{longtable} 
\usepackage[all,cmtip]{xy} 
 \usepackage{tikz}
 \usetikzlibrary{shapes,arrows,chains,matrix,positioning,scopes, patterns}
 \usepackage{pgfplots}

\usepackage{sectsty}
\allsectionsfont{\sffamily\mdseries\upshape} 

\usepackage{graphicx}
\usepackage[all]{xy}
\usepackage{epsfig}
\usepackage{graphics}
\usepackage{caption}
\usepackage{subcaption}
\usepackage{enumerate}
\usepackage{array}
\newcolumntype{P}[1]{>{\raggedright\arraybackslash}p{#1}}

\usepackage[nottoc,notlof,notlot]{tocbibind} 

\numberwithin{equation}{section}
\theoremstyle{plain}%
\newtheorem{theorem}{Theorem}
\numberwithin{theorem}{section}
\newtheorem{proposition}[theorem]{Proposition}
\newtheorem{example}[theorem]{Example}
\newtheorem{lemma}[theorem]{Lemma}
\newtheorem{corollary}[theorem]{Corollary}
\newtheorem{definition}[theorem]{Definition}
\newtheorem{remark}[theorem]{Remark}

\newtheorem{problem}[theorem]{Problem}
\newtheorem{algorithm}[theorem]{Algorithm}

\title{Matrix Completion for the Independence Model}
\author{Kaie Kubjas, Zvi Rosen}
\date{\today} 

\begin{document}
\maketitle

\begin{abstract}

We investigate the problem of completing partial matrices 
to rank-one matrices in the standard simplex $\Delta^{mn-1}$. 
The motivation for studying this problem comes from statistics: 
A lack of eligible completion can provide a falsification test for 
partial observations to come from the independence model. 
For each pattern of specified entries,
we give equations and inequalities which are satisfied 
if and only if an eligible completion exists.  We also describe
the set of valid completions, and we optimize over this set.

\noindent {\bf Key words:} matrix completion; independence model; weighted graphs; tensor completion; real algebraic geometry; optimal completions.

\noindent {\bf AMS subject classifications:} 15A83; 05C50; 14P10.
\end{abstract}

\section{Introduction}

The pattern $S$ of a partial matrix is the set of positions 
of specified entries in the partial matrix. A partial matrix $M$ is 
subordinate to $S$ if its pattern is $S$. 
We are interested in the geometry of two semialgebraic
sets: (1) the projection of rank-one matrices in $\Delta^{mn-1}$ 
to the entries in $S$, and (2) the set of 
completions for a fixed partial matrix $M$, as a subset of $\Delta^{mn-1}$.
Our exposition will be aimed at addressing the following two problems:

\begin{problem} \label{prob:image}
Given $m,n$, and $S \subseteq [m] \times [n]$, 
find defining equations and inequalities for the 
set of partial matrices subordinate to pattern $S$ that
can be completed to a rank-one matrix in $\Delta^{mn-1}$.
\end{problem}

\begin{problem}  \label{prob:fiber}
Given a partial matrix $M$ and $S \subseteq [m] \times [n]$, characterize 
all rank-one matrices in $\Delta^{mn-1}$ whose projection to the entries in $S$
agree with $M$.
\end{problem}

\begin{example} Let $m = n = 4 $, and $S =\{(1,1),(2,2),(3,3),(4,4)\}$. 
 By our results in later sections, we can answer 
 Problem \ref{prob:fiber} for the partial matrix on the left: there is a unique completion
 given by the matrix at right:
\[
\begin{pmatrix}
0.16 &  &  &  \\
 & 0.09 &  &  \\
 & & 0.04  &  \\
&  &  & 0.01  \\
\end{pmatrix} \to 
\begin{pmatrix}
0.16 & 0.12 & 0.08 & 0.04 \\
0.12 & 0.09 & 0.06 & 0.03 \\
0.08 & 0.06 & 0.04  & 0.02 \\
0.04 & 0.03 & 0.02 & 0.01  \\
\end{pmatrix}
\]
Perturbing any entry of the partial matrix by 
$\epsilon > 0$ makes the set of completions empty, and perturbing 
any entry by $\epsilon < 0$ 
introduces an infinite number of completions.
We will also answer Problem \ref{prob:image} for this
choice of $m,n$, and $S$: $M$ is completable
if and only if $\sum_{i = 1}^4 \sqrt{m_{ii}} \leq 1$.
\end{example}

The low-rank matrix completion is very well-studied: the three main directions have been convex relaxation of the rank constraints~\cite{fazel2001rank,CR09, CT10,recht2011simpler}, spectral matrix completion~\cite{KMO10} and algebraic-combinatorial approach~\cite{CJRW89,HHW06,KiralyTheranTomiokaUno}. We make use of existing algebraic-combinatorial
approaches in the latter three articles in our analysis. Our contribution is
looking at how the two conditions -- (1) nonnegativity and (2) 
summing to one -- affect the rank-one completion problem. The perspective of
the paper is a combinatorial and geometric rather than an algorithmic one, but we 
examine algorithms as they relate to the geometry.

The motivation for restricting to the simplex comes from statistics,
and was suggested by Vishesh Karwa and Aleksandra Slavkovi\'c.  
Let $X$ and $Y$ be two discrete random variables with $m$ and $n$ 
states respectively. Their joint probabilities are recorded in the matrix:
$
\mathbf{P} = (p_{ij}):1\leq i \leq m, 1 \leq j \leq n$, where  $ p_{ij} = Pr(X = i, Y = j)$.
 For any such matrix $\mathbf{P}$, we have $p_{ij} \geq 0$ for all $i,j$ and 
 $\sum p_{ij}=1$. We say that random variables $X$ and $Y$ are independent, if 
$ Pr(X = i, Y = j) = Pr(X = i)\cdot Pr(Y =j) $ for all $i,j$. This can be translated into the statement 
\[ \begin{matrix}
\mathbf{P} & = & \mathbf{X}^T \mathbf{Y }, \text{  where  }\\
\mathbf{X} & = & \begin{pmatrix}
Pr(X=1) & Pr(X=2) & \cdots &Pr(X=m) 
\end{pmatrix} \text{  and  } \\
\mathbf{Y} & = &
\begin{pmatrix}
Pr(Y=1) & Pr(Y=2) & \cdots & Pr(Y=n)
\end{pmatrix}.
\end{matrix}
\]
Hence, the matrix $\mathbf{P}$ of joint probabilities of two 
independent random variables has rank one, is 
nonnegative, and its entries sum to one. In other 
words, $\mathbf{P}$ is a rank-one matrix in the standard simplex $\Delta^{mn-1}$.

Our problems are of interest when probabilities $Pr(X=i,Y=j)$ 
are measurable only for certain pairs $(i, j)$. Situations in 
which this might arise in applications are: a pair of compounds 
in a laboratory that only react when in certain states, a pair
of alleles whose effects cancel each other out, etc. 
A complete answer to Problem~\ref{prob:image} will allow us to 
reject a hypothesis of independence of the events $X$ and $Y$, 
based only on this collection of probabilities. For other problems about 
matrix completion coming from statistics see
~\cite{King97} and~\cite{SU10}.
In the rest of the paper we will not consider the statistical context. 
Addressing questions like noise, the details of the 
falsification test, and experiments are left for a more 
statistical paper in the future.

\textbf{Outline.}
In Section~\ref{section:completability}, we derive for given $m,n$, and $S \subseteq [m] \times [n]$ 
inequalities and equations which are fulfilled by the entries 
of a partial matrix subordinate to $S$ if and only if the 
partial matrix is completable to a rank-one matrix in $\Delta^{mn-1}$. 
Our discussion starts with positive diagonal  partial matrices in Section~\ref{subsection:diagonal_masks},  continues with positive block partial matrices in Section~\ref{subsection:block_matrices} and positive general partial matrices 
in Section~\ref{subsection:arbitrary_masks}. 
Finally, Section~\ref{subsection:boundary} extends 
these results to nonnegative general  partial matrices.
The main result for positive partial matrices is described in
Theorem~\ref{main_theorem} and for nonnegative partial
matrices in Theorem~\ref{main_theorem2}.
The Section~\ref{section:completability} ends with an algorithm 
for checking completability of a partial matrix to a rank-one
matrix in $\Delta^{mn-1}$.

In 
Section~\ref{section:completions}, we describe the various
completions for a given partial matrix.  In 
Section~\ref{subsection:pos_completions}, 
we show how to construct a completion for a positive partial matrix.
In Section~\ref{subsection:optimization}, we will use Lagrange 
multipliers to construct a rank-one  completion in 
the standard simplex which maximizes or minimizes 
a certain function, e.g. the distance from the uniform distribution. 
In Section~\ref{subsection:boundary2}, we describe the set of
completions when the partial matrix contains zeros. 

In 
Section~\ref{section:generalizations}, we study generalization 
of our results to higher rank matrices and tensors in the standard simplex.
In particular, in Theorem~\ref{thm:tensors} we will derive a characterization 
of diagonal partial tensors which can be completed to 
rank-one  tensors in the standard simplex, i.e. rank-one tensors 
whose entries are nonnegative and sum to one. 

Partial tensors subordinate to a pattern $S$ which can be 
completed to a rank-one  tensor in the standard simplex 
form a semialgebraic set, see 
Proposition~\ref{prop:completable_matrices_form_a_semialgebraic_set}.
In 
Proposition~\ref{prop:semialgebraic_description}, we 
study the algebraic boundary of this semialgebraic set.

Implementations of algorithms can be found on
\[ \hbox{\tt math.upenn.edu/$\sim$zvihr/probCompletion.html} \]

\section{Completability : Equations and Inequalities}
\label{section:completability}

In this section, we completely solve 
Problem~\ref{prob:image} for all patterns $S$.
We begin by solving the problem for diagonal patterns, and
then extrapolate to cases of increasing generality:
\centerline{Diagonal $\implies$ Block $\implies$ General.}
Initially, we will only consider positive partial matrices, 
and find its defining equations and inequalities.
In Section~\ref{subsection:boundary}, we will consider partial
matrices with zero entries.

\subsection{Diagonal Patterns}
\label{subsection:diagonal_masks}

A \textit{diagonal partial matrix} is an $n \times n$ partial matrix (a point in $\mathbb{R}^n$) subordinate to the diagonal pattern $\{(1,1), \ldots, (n,n)\}$. In the $1 \times 1$ case, this is trivially completable, indeed completed, if and only if the observed entry is one. For a $2 \times 2$ matrix, there is more to consider.

\begin{example} \label{2by2}
Let $M$ be the partial matrix given by:
$M = \operatorname{diag}\left(
a , b \right)$.
In order for the matrix to be completed, both the rank one 
requirement and the summing to one must be addressed. 
First, for rank one, the off-diagonal entries are set to $x$ 
and $ab/x$, then the quantity $a +  ab/x + x + b$ is set 
equal to one. The equivalent quadratic equation is $x^2 + (a + b - 1) x + ab = 0$. In order for a real solution for $x$ to exist, the discriminant must be $\geq 0$, i.e.
\[ (a + b - 1)^2 - 4ab \geq 0.\]
This inequality, along with the requirement that $a + b \leq 1$ and both $a, b \geq 0$, is necessary and sufficient to guarantee that $x$ gives a completion in $\Delta^3$, see Figure~\ref{figure:2x2}.

\begin{figure}[h!]
\centering
\includegraphics[width=.5\textwidth]{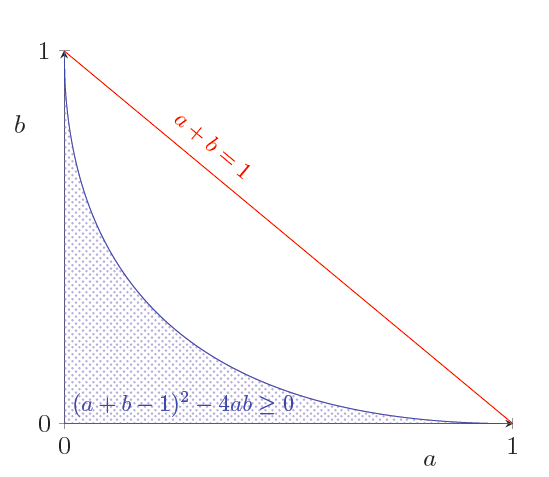}
\caption{\small{$2 \times 2$ diagonal partial matrices which can be completed to a rank-one matrix in $\Delta^3$.}}
\label{figure:2x2}
\end{figure}

\end{example}

For $n > 2$, we take advantage of the factorization of rank-one matrices as products of vectors to obtain the following more general result:

\begin{theorem} \label{halfnorm}
Let $M$ be an $n \times n$ diagonal partial matrix given by $\operatorname{diag}(a_1,\ldots,a_n)$, where $n\geq 2$, and $a_i > 0$ for all $i$. Then $M$ is completable if and only if $\sum_{i = 1}^n \sqrt{a_i} \leq 1$, or equivalently, $|| (a_1,\ldots,a_n) ||_{1/2} \leq 1$. 
\end{theorem}

\begin{proof}
Recall that a rank-one $n\times n$  matrix $M \in \Delta^{nn-1}$ can be factored as ${\bf u}^T{\bf v}$ for ${\bf u}, {\bf v} \in \Delta^{n-1}$. For this problem, consider all possible values of ${\bf u}$ in $\Delta^{n-1}$, but do not restrict the values of ${\bf v}$ to the simplex. Instead, let ${\bf v}$ be formulated in terms of ${\bf u}$ and the entries of the matrix.  Explicitly, we have $M = \begin{pmatrix}
u_1 &  \cdots & u_n
\end{pmatrix}^T
\begin{pmatrix} 
a_1/u_1 & \cdots & a_n/u_n 
\end{pmatrix}$.
An eligible completion will arise when $\sum v_i = 1$.  Here we assume $a_i>0$ and thus $u_i>0$ for all $i$.

Let $f({\bf u})=\sum_{i=1}^n v_i=\sum_{i=1}^n a_i/ u_i$ and compute the minimum of $f$ on the simplex. For this computation, consider $u_1,\ldots, u_{n-1}$ as independent variables and $u_n = 1- \sum u_i$.

\[ f = \left(\sum_{i = 1}^{n-1} \dfrac{a_i}{u_i} \right) + \dfrac{a_n}{1-\sum_{i=1}^{n-1} u_i} \:
\Rightarrow \:\: \dfrac{\partial f}{\partial u_i} = -\dfrac{a_i}{u_i^2} + \dfrac{a_n}{(1-\sum_{i=1}^{n-1} u_i)^2}\]
Setting $\partial f/ \partial u_i = 0$ for all $i=1,\ldots,n-1$ implies $a_i/u_i^2=k$ is constant for all $i$. Since ${\bf u}$ is in the simplex, we have $k$ equal to $(\sum \sqrt{a_i})^2$. The value of $f$ (i.e. the sum of $v_i$) at this point is $(\sum \sqrt{a_i})^2$.  If this is $\leq 1$, continuity of $f$ implies that a completion exists somewhere between our minimum and the boundary, because within an $\epsilon$ of the boundary of $\Delta^{n-1}$, we have
\[ \sum v_i =  \dfrac{a_1}{u_1} + \cdots + \dfrac{a_n}{u_n} \gg 1.\]
If this minimum value is $> 1$, the function will not achieve $1$ anywhere in the simplex, so no completion is possible.
\end{proof}

\begin{remark}\label{remark:2x2_case_example_vs_theorem}
The result of Theorem \ref{halfnorm} works to derive the inequality for the $2\times 2$ diagonal partial matrices as in Example \ref{2by2}:
\[ \begin{matrix} \sqrt{a} + \sqrt{b} \leq 1 & \Leftrightarrow &a + b + 2\sqrt{a b} \leq 1 & \Leftrightarrow & 2\sqrt{a b} \leq 1 - a - b \\
& \Leftrightarrow & 4 a b \leq (1 - a - b)^2 & \text{ and } & 0 \leq 1-a-b. \end{matrix}\]
\end{remark}

The set of $3 \times 3$ diagonal partial matrices that are completable to  rank-one matrices in the standard simplex is shown in Figure~\ref{figure:3x3}.
\begin{figure}[h]
\centering
\includegraphics[scale = 0.4]{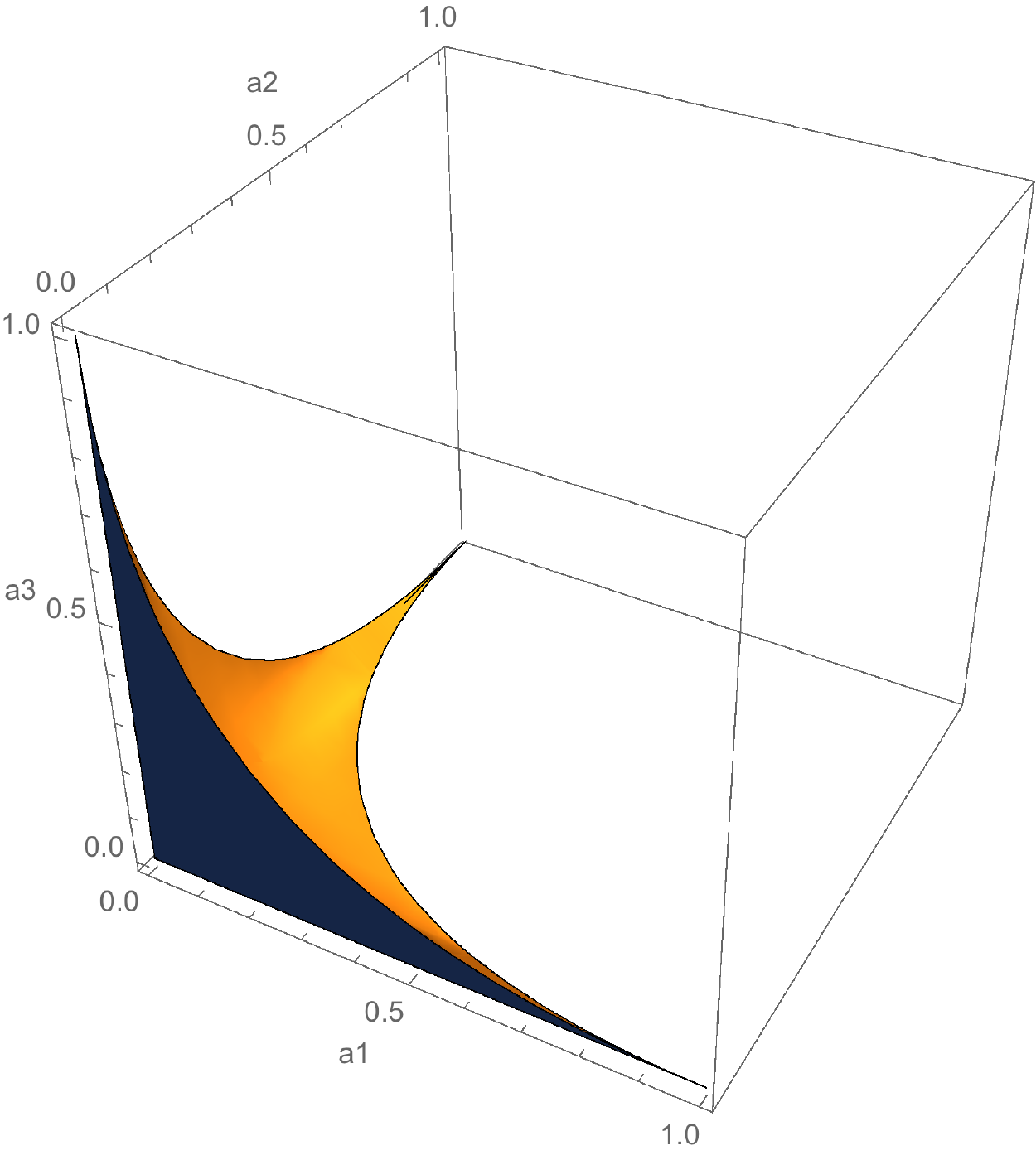}
\caption{\small $3 \times 3$ diagonal partial matrices which are completable to a rank-one matrix in $\Delta^{8}$.}
\label{figure:3x3}
\end{figure}

\subsection{Block Diagonal Patterns}\label{subsection:block_matrices}

We say that $\{i_1, \ldots, i_p\} \times \{j_1, \ldots, j_q\}$ 
is a block of specified entries if the entries specified in 
rows $\{i_1, \ldots, i_p\}$ and columns $\{j_1, \ldots, j_q\}$ 
are exactly $\{i_1, \ldots, i_p\} \times \{j_1, \ldots, j_q\}$.  
A \textit{block partial matrix} is a partial matrix subordinate 
to a pattern that is the union of blocks of specified entries 
and every row and column contains at least one specified entry.
We begin with a corollary to Theorem~\ref{halfnorm}.

\begin{corollary}\label{cor:completability_block}
Let $M$ be a partial matrix with a rank-one block of positively specified entries 
$\{1, \ldots, p\} \times \{1, \ldots, q\}$. 
Let $\widehat{M}$ be the partial matrix obtained from $M$ 
by contracting rows $1, \ldots, p$ to one row, columns 
$1, \ldots, q$ to one column and specifying 
$$\widehat{m}_{11}=\displaystyle \sum_{i = 1}^p \sum_{j = 1}^q m_{ij}.$$ 
The completability of $M$ is equivalent to completability of $\widehat{M}$.
\end{corollary}

\begin{proof}
($\Rightarrow$) Let ${\bf u}^T {\bf v}$ be a completion of $M$. Then one completion of $\widehat{M}$ is:
\vspace{-2mm}
\[
\begin{pmatrix}
\sum_{i=1}^p u_i &
u_{p+1} &
\cdots &
u_m
\end{pmatrix}^T
\begin{pmatrix}
\sum_{j=1}^q v_j & v_{q+1} & \cdots & v_n
\end{pmatrix}.
\] \vspace{-6mm}

($\Leftarrow$) Let $\widehat{\bf {u}}^T \widehat  {\bf v}$ be a
 completion of $\widehat M$, where 
$\widehat{\bf u}=(\widehat{u}_1,
 \widehat{u}_{p+1}, \ldots, \widehat{u}_m)$ \newline and $\widehat{\bf v}=
 (\widehat{v}_1,\widehat{v}_{q+1},$ $\ldots, \widehat{v}_n)$. Since 
 $m_{1,1},\ldots,m_{p,q}$ form a rank-one $p\times q$ matrix, 
 it has a rank-one factorization ${\bf u'}^T {\bf v'}$. Then one completion of $M$ is:
\newline
$
\begin{pmatrix}
\widehat{u}_1 \frac{u'_1}{\sum u'_i} &
\cdots&
\widehat{u}_1 \frac{u'_p}{\sum u'_i}&
\widehat{u}_{p+1}&
\cdots&
\widehat{u}_m
\end{pmatrix}^T
\begin{pmatrix}
\widehat{v}_1 \frac{v'_1}{\sum v'_i} &
\cdots &
\widehat{v}_1 \frac{v'_q}{\sum v'_i} &
\widehat{v}_{q+1} &
\cdots &
\widehat{v}_m
\end{pmatrix}.
$
\end{proof}

The corollary implies that a rank-one block in the upper-left corner
of a matrix can be replaced by the sum of its entries without changing
the completability of the full matrix. The same logic applies to any
rank-one block in the matrix, since row and column permutations do not
affect rank or the value of the sum of entries. We then have this
result for block partial matrices.

\begin{corollary}
Let $M$ be a block partial matrix with rank-one blocks of positively specified entries. Let $\widehat{M}$ be a 
diagonal partial matrix with one diagonal entry for each
 block of specified entries of $M$. Let a diagonal entry of
 $\widehat{M}$ be equal to the sum of entries in the 
corresponding block of specified entries of $M$. 
Then completability of $M$ is equivalent to 
completability of $\widehat{M}$.
\end{corollary}

\subsection{General Partial Matrices}\label{subsection:arbitrary_masks}

To discuss general partial matrices, we introduce graph notation used in the matrix completion literature, including for example \cite{KiralyTheranTomiokaUno}.

\begin{definition}
Let $M = {\bf u}^T {\bf v}$ be a matrix 
in $\Delta^{mn-1}$, and let ${\bf u}$ 
and ${\bf v}$ be vectors in $\Delta^{m-1}$ and $\Delta^{n-1}$ 
respectively. A bipartite graph can be associated to $M$ in the following way:

\vspace{5mm}

\centerline{
\begin{tabular}{|l|l|} \hline
{\bf Graph }& {\bf Matrix} \\ \hline
White vertex $r_i$ & $i$-th row \\
Black vertex $c_j$ & $j$-th column \\
Edge $(r_i c_j)$ & $(i,j)$-th entry \\
Weight $\omega(r_i)$ & Sum of $i$-th row, or $u_i$ \\
Weight $\omega(c_j)$ & Sum of $j$-th column, or $v_j$ \\
Edge weight $\omega(r_i c_j)$ & Value of entry $m_{ij}$ \\\hline
\end{tabular}
}
\vspace{3mm}
The bipartite graph $G(M)$ associated to a 
partial matrix $M$ is the graph obtained by 
deleting the edges corresponding to unobserved entries, 
and omitting vertex weights. If we want to talk about 
the corresponding unweighted graph, 
then we say the bipartite graph $G(S)$ associated to a pattern $S$.
\end{definition}

\begin{example}\label{example:probability_mask}
On the left is a partial matrix and on the right is the corresponding bipartite graph.

\begin{minipage}{.5\textwidth}
\begin{equation*}
\begin{pmatrix}
m_{11} & m_{12} & \\
 & m_{22} & \\
& & m_{33}
\end{pmatrix}
\end{equation*}
\end{minipage}
\begin{minipage}{.5\textwidth}
\includegraphics[scale=.8]{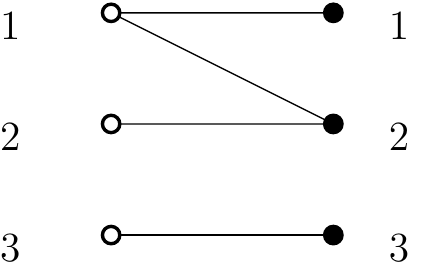}
\end{minipage}

\end{example}

In this formulation, the question of completability 
is equivalent to the existence of a vertex labeling so that 
the black vertex weights and white vertex weights each 
sum to one, and the edge weights satisfy $\omega(r_i c_j) 
= \omega(r_i) \omega(c_j)$. 
The edge set of the bipartite graph associated to a diagonal 
partial matrix forms a perfect matching of $K_{n,n}$. 
The bipartite graph associated to a block partial matrix 
is the union of disjoint complete bipartite graphs. 
We now consider more general partial matrices and 
bipartite graphs. For this we introduce a definition
and a lemma paraphrased from~\cite[Section 6]{CJRW89}:

\begin{definition}[\cite{CJRW89}]
Let $S$ be a pattern with corresponding bipartite graph 
$G(S)$. Call $C = \{(r_1,c_1),(r_1,c_2),(r_2,c_2),
\ldots,(r_k,c_k),(r_k,c_1)\} \subseteq S$ a cycle 
in $G(S)$ if $G(C)$ is a cycle in $G(S)$. 

Let $M$ be a partial matrix with pattern $S$, 
and let $m_{rc}$ be the entry in $M$ corresponding 
to the edge $(r,c)$ in $G(S)$. The partial matrix $M$ 
will be called singular with respect to a cycle $C$ in $G(S)$ if
\begin{equation}
m_{r_1 c_1} m_{r_2 c_2} \ldots m_{r_k c_k} = 
m_{r_1 c_2} m_{r_2 c_3} \ldots 
m_{r_{k-1} c_k} m_{r_k c_1}. \label{binomial}
\end{equation}
\end{definition}

\begin{lemma}[\cite{CJRW89}, Lemma 6.2]\label{lemma:rank_one_completion}
Let $M$ be a partial matrix subordinate to a pattern $S$ 
with all specified entries positive. Then the minimal 
rank over all completions of $M$ is one if and 
only if $M$ is singular with respect to all cycles in $G(S)$.
\end{lemma}

In Lemma~\ref{lemma:rank_one_completion}, 
a completion means any completion and not 
necessarily a rank-one completion in the standard 
simplex as in the rest of the paper. 
Lemma~\ref{lemma:rank_one_completion} 
gives a necessary condition for a rank-one completion 
to exist. 

The algorithm for completing the uniquely completable 
entries in a rank-one completion is given implicitly in the 
proof of~\cite[Lemma 6.2]{CJRW89} and stated explicitly in~\cite[Formula 6]{HHW06},~\cite[Theorem 34]{KiralyTheranTomiokaUno} and~\cite[Theorem 2.6]{KiralyTheran}. 
We describe the algorithm as it is written in the proof of~\cite[Lemma 6.2]{CJRW89}: 

\begin{algorithm}
Let $m_{i j}$ be an unspecified entry. Let
$ \{(r_1,c_1), (r_1, c_2), \ldots ,(r_k, c_k),$ $(r_k, c_1) \}$
be a cycle in $G(S \cup (i,j))$ which was not a cycle 
in $G(S)$. One of the edges in this cycle is $(i,j)$, 
say $(r_1,c_1)=(i,j)$. Let
$p:=m_{r_2 c_2} m_{r_3 c_3} \ldots m_{r_k c_k},$ and
$q:=m_{r_1 c_2} m_{r_2 c_3} \ldots m_{r_k c_1}$. 
Then, set
\begin{equation}
m_{i j} =\frac{q}{p}=\frac{m_{r_1 c_2} m_{r_2 c_3} 
\ldots m_{r_k c_1}}{m_{r_2 c_2} m_{r_3 c_3} \ldots 
m_{r_k c_k}}. \label{rational_function}
\end{equation}
\end{algorithm}

The following corollary in~\cite{HHW06,KiralyTheranTomiokaUno,
KiralyTheran} establishes which entries of a partial matrix 
are uniquely reconstructible by cycles if the 
partial matrix does not contain zeros:

\begin{corollary}[\cite{HHW06}, Corollary 6, \cite{KiralyTheranTomiokaUno}, 
Theorem 34, \cite{KiralyTheran}, Theorem 2.4]
\label{thm:rank_one_unique_completability}
\label{thm:completable_entries}
Let $M$ be a partial matrix subordinate to a pattern $S$ with all specified entries non-zero. 
The set of uniquely reconstructible entries of $M$ 
is exactly the set $m_e$ with $e$ in the transitive 
closure of $G(S)$. In particular, all of $M$ is 
reconstructible if and only if $G(S)$ is connected.
\end{corollary}

The transitive closure $\overline{G}$ of a bipartite graph $G$ 
is the graph that is obtained from $G$ by 
replacing each connected component by a 
complete bipartite graph on the vertices that 
belong to the connected component. 
Suppose that a partial matrix $M$ has all positive 
entries and is singular with respect to cycles in $G(S)$, 
the completability of $M$ is equivalent to the 
completability of $M$ together with its uniquely 
reconstructible entries. By Corollary~\ref{thm:rank_one_unique_completability},
this is a partial matrix that consists of 
non-zero blocks.

\begin{theorem}\label{main_theorem}
Let $X$ be the set of $m\times n$ positive partial matrices 
$M$ subordinate to the pattern $S$ that can be completed to a rank-one matrix
in $\Delta^{mn-1}$.
Define entries corresponding to the edges in the transitive closure $\overline{G(S)}$
using Formula~(\ref{rational_function}).
Let $s$ be the number of connected components of $G(S)$ that contain at least one edge,
let $K_i$ be the $i$-th block of $\overline{G(S)}$, and let $b_i$ 
be the sum of the weights in $K_i \subset \overline{G(S)}$. 
\vspace{-2mm}
\begin{enumerate}  \setlength{\itemsep}{-5pt}
\item If $\overline{G(S)}=K_{m,n}$, the defining constraints for $X$
are: \begin{enumerate} \setlength{\itemsep}{0pt} \vspace{-2mm}
 \item the ideal generated by relations~(\ref{binomial}) corresponding to cycles of $S$, and $b_1 = 1$; and
 \item the inequalities $m_{ij} > 0$ for all $(i,j) \in S$.
 \end{enumerate}  
 \item Otherwise the defining constraints for $X$
are: \begin{enumerate} \setlength{\itemsep}{0pt} \vspace{-2mm}
 \item the ideal generated by relations~(\ref{binomial}) corresponding to cycles of $S$; and
 \item the inequalities $m_{ij} > 0$ for all $(i,j) \in S$, and $\sum_{l = 1}^s \sqrt{b_l}\leq 1$.
 \end{enumerate}
\end{enumerate}
\end{theorem}

\begin{proof}
In both parts the ideal of relations corresponding to cycles of $S$
ensures that the conditions of Lemma~\ref{lemma:rank_one_completion} are satisfied. By Corollary~\ref{thm:rank_one_unique_completability} and
Corollary~\ref{cor:completability_block}, we can reduce to the diagonal partial matrix
case with possibly some empty rows and columns.
A completion of the matrix with empty rows and columns removed gives a completion of the 
original matrix by setting all the entries in the empty rows and columns equal to zero. 
A completion of the original matrix gives a completion of the matrix with empty rows and columns removed
with the sum of entries $\leq 1$. By the continuity argument as in the proof of Theorem~\ref{halfnorm}, 
this matrix has a desired completion. Finally we apply Theorem~\ref{halfnorm}. 
\end{proof}

Note that these defining equations and inequalities include
rational functions and algebraic functions like square root.
The Tarski-Seidenberg Theorem ensures that these can
be converted into polynomial equations and inequalities; we will
explore this in more detail in Section~\ref{subsection:semialgebraic_set}.

\begin{example}
The partial matrices of pattern $S = \{(1,1),(1,2),(2,1),(3,3)\}$ 
with all observed entries positive has a completion if and only if
\[\sqrt{m_{11}+m_{12}+m_{21}+m_{12}m_{21}/m_{11}}+\sqrt{m_{33}} \leq 1.\]
This is equivalent to the conditions
\begin{align}
(m_{11}+m_{12}+m_{21}+m_{12}m_{21}/m_{11}+m_{33}-1)^2
-4(m_{11}+m_{12}+m_{21}+m_{12}m_{21}/m_{11})m_{33} & \geq  0, \nonumber\\
m_{11}+m_{12}+m_{21}+m_{12}m_{21}/m_{11}+m_{33} & \leq  1.\nonumber
\end{align}
By clearing the denominators, we get polynomial 
inequalities in the observed entries whose solutions 
are all completable partial matrices of pattern $S$.
\end{example}

\subsection{Boundary of the Set of Completable Matrices} \label{subsection:boundary}

As promised, we discuss the set of partial matrices
which have some specified entries equal to zero. Fix a pattern $S$ and let $Z \subset S$ be the subset of
entries set to zero. Specifically, $m_{ij} = 0$ for $(i,j) \in Z$, and 
$m_{ij} > 0$ for $(i,j) \notin Z$.

\begin{proposition}
A diagonal partial $n\times n$ matrix $\text{diag}(a_1, \ldots ,a_n)$ is completable to a rank-one matrix in $\Delta^{nn-1}$ if and only if
$a_{i} \geq 0$ for $i = 1,\ldots,n$ and $\sum_{i = 1}^n \sqrt{a_{i}} \leq 1$.
\end{proposition}

\begin{proof} Necessity of the conditions follows by the Cauchy-Schwarz
inequality; sufficiency is
given case by case:
For $|Z| \leq n-2$, all rows and columns containing zeros may be
set to zero, and the remaining submatrix can be completed by Theorem
\ref{halfnorm}.  For $|Z| = n-1$, without loss of generality 
take $a_{1} > 0$. Set
\vspace{-2mm}
\[ {\bf u} = \begin{pmatrix} 1 & 0 & \cdots & 0 \end{pmatrix}, \text{  and  }
 {\bf v} = \begin{pmatrix} a_{1} & 1- a_1 & 0 &\cdots & 0 \end{pmatrix}. \]
 \vspace{-7mm}
 
\noindent ${\bf u}^T{\bf v}$ will be a completion of the matrix.
For $|Z| = n$, simply set some off-diagonal entry to one and all other
entries to zero.
\end{proof}

The strategy used in this proof is our general approach to dealing
with zeros. We reduce to a smaller positive submatrix by removing
some rows and columns. We use the same approach to pass to
general submatrices; first we cite another definition from
\cite{CJRW89}.

\begin{definition}[\cite{CJRW89}]
We say that $L \subseteq S$ is a 3-line in $G(S)$ if $L$ is of the form
$$\{(r_1,c_1),(r_1,c_2),(r_2,c_1)\}$$
where $r_1 \neq r_2, c_1 \neq c_2$ and $(r_2,c_2) \not \in S$. 
In other words, $G(L)$ is a line consisting of three edges in $G(S)$.
The matrix will be called singular with respect to a 3-line 
$L$ in $G(S)$ if either $m_{r_1 c_1} \neq 0$ or
 $m_{r_1 c_2} m_{r_2 c_1}=0$ (or both). This is equivalent to the 
\textit{zero row or column property} in~\cite{HHW06} that states that if any 
entry of a rank-one matrix $M$ is zero, then either every other entry
of $M$ in the same row is zero or every other entry in the same column
is zero.
\end{definition}

\begin{lemma}[\cite{HHW06}, Lemma 1]\label{lemma:zeros}
If a partial matrix has a rank-one completion, then it is singular with 
respect to 3-lines or equivalently it has the zero row or column property.
\end{lemma}

\begin{definition}
Let $S \subset [m] \times [n]$, $I \subset [m]$ and $J \subset [n]$. 
Let $X(I,J;S)$ be the subset of $\mathbb{R}^S$
such that: \begin{enumerate} \setlength{\itemsep}{-2pt}
\item for $(i,j) \in S$, $m_{ij} = 0$ if and only if $i \in I$ or
$j \in J$, and 
\item for the pattern $S'$ induced by restricting $S$ to
$([m] \setminus I) \times ([n] \setminus J)$ 
the conditions of Theorem~\ref{main_theorem} are satisfied.
\end{enumerate}
\end{definition}

\begin{theorem} \label{main_theorem2}
Let $S\subset [m] \times [n]$ be a pattern of specified entries. 
Then the set of partial matrices subordinate to pattern $S$ which can be completed to a rank-one matrix
in $\Delta^{mn-1}$ is given by

\[
\bigcup_{I,J} \: X(I,J;S) :  I \subsetneq [m], J \subsetneq [n].
\]
\end{theorem}
\begin{proof}
By Lemma~\ref{lemma:zeros}, a partial matrix $M$
is completable if there exist rows $I \subset [m]$ and columns $J
\subset [n]$ that contain all the zeros of the partial matrix
$M$  and the partial matrix obtained after removing the rows in $I$
and columns in $J$ is completable. Finally we apply Theorem~\ref{main_theorem}.
\end{proof}

\begin{example}
Let $S = \{ (1,1), (1,2), (2,2), (2,3) \}$ be a pattern as a subset
of $2 \times 3$ matrices. In Table \ref{table:boundary}, the
entry in row $I$ and column $J$ is a letter referring to
a defining equation or inequality (up to change of index).

\begin{table}[h!]
\[
\begin{array}{|c|ccccccc|} \hline
		   & \varnothing & \{1\} 	& \{2\} 	& \{3\} 	& \{1,2\} 	& \{1,3\} 	& \{2,3\} 	\\ \hline
\varnothing & A			& B		& C		& B		& D		& E		& D		\\
\{1\}		   & F			& E		& D		& D		& G		& G		& --		\\
\{2\}		   & F			& D		& D		& E		& --		& G		& G		\\ \hline
\end{array}
\]
\caption{Types of $X(I,J;S)$.}
\label{table:boundary}
\end{table} 

\vspace{-9mm}
\[
\begin{array}{cllcl}
A: & (m_{11} + m_{12} + m_{22} + m_{23} - 1)m_{12}m_{22} + m_{11}m_{22}^2 + m_{12}^2 m_{23} = 0& &E: & m_{12} + m_{22} = 1 \\
B: & (m_{12} + m_{22} + m_{23} - 1)m_{22} + m_{12} m_{23} = 0 & & F: & m_{22} + m_{33} \leq 1 \\
C: & \sqrt{m_{11}} + \sqrt{m_{23}} \leq 1& & G: & m_{23} = 1 \\
D: & m_{23} \leq 1 \\
\end{array}
\]

For instance, the entry in row $\varnothing$, column $\{2\}$
corresponds to setting no rows equal to zero and the 
second column equal to zero. The resulting set
is cut out by:
\[ m_{12} = m_{22} = 0, m_{11} > 0, m_{23} > 0 \text{ and } \sqrt{m_{11}} + \sqrt{m_{23}} \leq 1.\]

\end{example}
The geometry gives rise to an algorithm for deciding completability:

\begin{algorithm}[Decide completability for a partial matrix $M$] \hspace{2mm} \newline
\vspace{-6mm}

\begin{enumerate} \setlength{\itemsep}{-3pt}
\item Translate $M$ into the corresponding bipartite graph 
$G$ including edge weights.

\item Check whether $G$ is singular with respect to 3-lines and cycles. If a 3-line or a cycle fails, return ``NO''.

\item Execute rank-one completion using 3-lines. Remove vertices that are connected to every vertex in a partite set with weight zero.

\item Execute rank-one completion using cycles. 
Let the graph obtained after this step have $s$ connected components $C_1,\ldots, C_s$.

\item If $s=1$ and the edge weights add up to one, return ``YES". Else, return ``NO".

\item  Let $S = \sum_{i = 1}^s \sqrt{b_i}$ where $b_i$ is the sum of the entries in $C_i$. If $S > 1$, return ``NO". Else, return ``YES".
\end{enumerate}
\end{algorithm}

\begin{remark}
This algorithm runs in polynomial time in $O(m^2n^2)$. 
Feasibility check and completion by cycles for each entry of the matrix runs in linear time in the number of observed entries by \cite[Algorithm~2]{KiralyTheran}. Feasibility check and completion by 3-lines is $O(E(m+n))$, where $E$ is the number of observed entries: For each zero one checks whether there is a non-zero entry in the column and in the row containing it. If there is a non-zero entry in both, then the partial matrix is not singular with respect to 3-lines. If there is a non-zero entry in the row, then complete all the entries in the column to zero (and vice-versa). If neither row nor column contains non-zero entries, then do nothing.
\end{remark}

\section{The Set of Completions}

\label{section:completions}

Throughout Section \ref{section:completability}, we
needed to prove the existence of a rank-one completion 
in the standard simplex for a partial matrix. Now,
we assume existence of such a completion, and seek to describe
all completions.  For the beginning of this discussion,
we assume 
no coordinates of $M$ are
zero. 
This case will be addressed in Section \ref{subsection:boundary2}.

\subsection{Completions for Positive Partial Matrices} \label{subsection:pos_completions}

\begin{theorem} \label{theorem:completions}
Let $M$ be a completable partial matrix with positive entries,
subordinate to pattern $S$. 
Let $\overline{M}$ be the block-diagonal partial matrix obtained
by completing $M$ with Formula~(\ref{rational_function}). Let $s$ be
the number of blocks in $\overline{M}$. Let $b_i$ be
the sum of the weights in each block.
\begin{enumerate} \setlength{\itemsep}{-3pt}
\item if $\sum \sqrt{b_i} = 1$, then there is a unique completion;

\item if $\sum \sqrt{b_i} < 1,s = 2$ and one component is an isolated vertex,
there is a unique completion;
\item if $\sum \sqrt{b_i} < 1, s = 2$ and both components contain edges,
there are two completions;
\item if $\sum \sqrt{b_i} < 1, s > 2$, then there is an $(s-2)$-dimensional 
basic semialgebraic set of completions.
\end{enumerate}
\end{theorem}

\begin{proof}
\noindent \emph{\bf The unique completion for $\sum \sqrt{b_i} = 1$.}
The logic follows from the proof of Theorem~\ref{halfnorm}; the diagonal 
matrix is given by $(\sqrt{b_1} \cdots  \sqrt{b_n})^T(\sqrt{b_1} \cdots  \sqrt{b_n})$.
Since the diagonal version of the matrix has a unique solution there,
the completion of the original partial matrix is uniquely determined by it.

\noindent \emph{\bf The completion for $s = 2$ with an isolated vertex.}
Complete the non-empty block, and let the block sum be $b_1 < 1$.
Because one row or column is left, simply add a copy of the first row or
column and scale it until the sum is one.

\noindent \emph{\bf 
The two completions for $s = 2$ with two non-empty components.
}
Complete the two non-empty blocks. If $\sqrt{b_1} + \sqrt{b_2} < 1$, then
the factorization below
\[  \left(\begin{array}{cc} {(\sqrt{b_1} + \sqrt{b_2})}\sqrt{b_1} & {(\sqrt{b_1} + \sqrt{b_2})}\sqrt{b_2} \end{array} \right)^T
\left(\begin{array}{cc} \frac{\sqrt{b_1}}{\sqrt{b_1} + \sqrt{b_2}} &  \frac{\sqrt{b_2}}{\sqrt{b_1} + \sqrt{b_2}} \end{array} \right)\]
fails to keep the first factor in the simplex. Letting this vector vary
as the second factor moves to the edges of the simplex, the sum of the
coordinates is strictly increasing.  Using the formula in 
Proposition~\ref{thm:easycompletion}, we can find the two points
where the first factor lies in the simplex.

\noindent \emph{\bf The $(s-2)$-dimensional set of completions for $s > 2$.}
We start as in the $s=2$ case at the parametrization
\[  \left(\begin{array}{ccc} {(\sum \sqrt{b_i} )}\sqrt{b_1} & \cdots & {(\sum \sqrt{b_i})}\sqrt{b_s} \end{array} \right)^T
\left(\begin{array}{ccc} \frac{\sqrt{b_1}}{\sum \sqrt{b_i}} & \cdots & \frac{\sqrt{b_s}}{\sum \sqrt{b_i}} \end{array} \right)\]
Again the parametrization which matches our given values
fails to have a first factor in the simplex. As we move to the boundary
of the simplex along any direction, the sum in the first factor will hit $1$.
Therefore, the set of completions is parametrized by the sphere $S^{s-2}$.
\end{proof}

In the diagonal partial matrix case we will explicitly describe the semialgebraic set. One can use results in Section~\ref{section:completability} to derive semialgebraic descriptions in other cases. Let $M=\text{diag}(a_1, \ldots, a_n)$ and define $S=\sum _{i=1}^n\sqrt{a_i}$. Let us parametrize a vector ${\bf u} \in \Delta^{n-1}$ by
\[
{\bf u(t)}=\begin{pmatrix}
\frac{\sqrt{a_1}}{S}+t_1,\frac{\sqrt{a_2}}{S}+t_2, \ldots , \frac{\sqrt{a_{n-1}}}{S}+t_{n-1},\frac{\sqrt{a_n}}{S}+t_n
\end{pmatrix},
\]
where $t_n=-t_1-t_2-\ldots -t_{n-1}$.
Then
\[
f({\bf u(t)})=\sum_{i=1}^n v_i({\bf t})=\sum_{i=1}^n \frac{a_i}{u_i({\bf t})}  =\sum_{i=1}^{n}  \frac{a_i}{\frac{\sqrt{a_i}}{S}+t_i}= \sum_{i=1}^n \frac{a_i S}{\sqrt{a_i}+t_iS}.
\]

\begin{proposition}
The semialgebraic set of completions of a diagonal partial matrix $\text{diag}(a_1,a_2,\ldots ,a_n)$ is given by $f({\bf u(t)})=1$ and ${\bf u(t)}\geq 0$ (after clearing denominators). 
\end{proposition}

Any path from the local minimum to the boundary of the simplex will strike at least one solution. If any completion is acceptable, we can designate a simple path and find its points of intersection with the semialgebraic set of completions.

\begin{proposition} Let $M = \operatorname{diag}(a_1,\ldots,a_n)$, such that $n > 2$ and $S  = \sum \sqrt{a_i} < 1$. Then, a completion of $M$ is given by:
\[ {\bf u} = \left(\frac{\sqrt{a_1}}{S} + t, \frac{\sqrt{a_2}}{S} - t, \frac{\sqrt{a_3}}{S},\ldots, \frac{\sqrt{a_n}}{S}\right),\]
where $t$ is one of the solutions to the following quadratic equation:
\begin{equation}
S^2\left(S(\sqrt{a_1} + \sqrt{a_2})-S^2+1\right) t^2 
 - \: S(S^2-1)\left(\sqrt{a_1} - \sqrt{a_2}\right) t
+\:  (S^2-1) \sqrt{a_1 a_2} = 0, \: \: \:  \label{eq:walkparameter}
\end{equation}
both of which lie in the interval $[- \sqrt{a_1}/S, \sqrt{a_2}/S]$.
\label{thm:easycompletion}
\end{proposition}

\begin{proof}
The trajectory traced for values of $t \in [- \sqrt{a_1}/S, \sqrt{a_2}/S]$, is a line segment on the simplex. Setting the sum of the coordinates of $( a_i / u_i)_{i = 1,\ldots,n}$ equal to one and clearing denominators gives the quadratic equation above. Since it passes through the local minimum, continuity implies existence of two solutions in the desired interval.
\end{proof}

\begin{example}
Consider the matrix $M = \operatorname{diag}(1/4,1/25,1/36)$. To obtain a completion, one may solve the quadratic equation~(\ref{eq:walkparameter}), which after scaling turns into
\[ 9295t^2+936t-360 = 0, \]
giving solutions $t = -\frac{6}{715} \, \sqrt{586} - \frac{36}{715}$, or 
$\frac{6}{715} \, \sqrt{586} - \frac{36}{715}$. Using the latter value, we obtain the matrix completion:
\[ {\bf u}^T {\bf v} = \left(\begin{array}{rrr}
0.250 & 0.374 & 0.105 \\
0.027 & 0.040 & 0.011 \\
0.066 & 0.098 & 0.028
\end{array}\right).\]
\end{example}

\subsection{Optimization on a Completion Set}
\label{subsection:optimization}

When the graph has $s \geq 2$ connected components which contain at least an edge and $\sum_{i = 1}^s \sqrt{b_i} < 1$, there is an $(s-2)$-dimensional set of completions. For example, consider a diagonal partial $3 \times 3$ matrix with each observed entry equal to the same constant $c$. In Figure \ref{fig:levels}, each curve represents values of ${\bf u}$ that parametrize a completion of the partial matrix with $c$ on the diagonal, for various values of $c$. Here ${\bf u}$ is projected onto the first two coordinates. 

\begin{figure}[h]
\centerline{
\includegraphics[scale=.4]{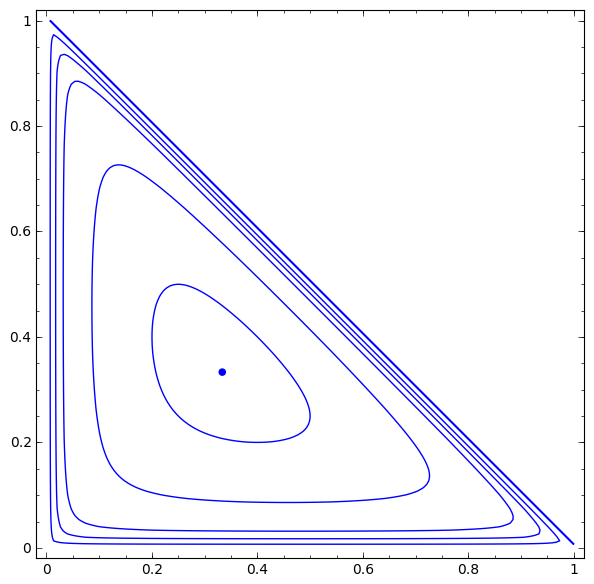}}
\caption{\small Solution curves for $c = 1/9, 1/10, 1/16, 1/36,1/64$, and $1/150$, in the projection of the simplex onto the first two coordinates.}
\label{fig:levels}
\end{figure}

If we have infinitely many completions, 
one way to find a desired completion is to minimize or maximize a distance 
measure $d$ from a fixed matrix in the standard 
simplex. We will explain how to use Lagrange multipliers to solve this optimization problem if $d$ is the Euclidean distance from the uniform distribution.

By the method of Lagrange multipliers, an element $\begin{pmatrix} t_1, t_2, \cdots  ,t_{n-1}\end{pmatrix}$ in this semialgebraic set is a critical point for a distance function $d$ if and only if the gradient of $d$ is a constant multiple of the vector of partial derivatives $\frac{\partial f}{\partial t_i}$. To compute all the critical points of the function $d$ on the variety given by $f({\bf t})=1$, we need to solve the system of rational equations given by $f({\bf t})=1$ and all the $2 \times 2$ minors of the matrix 
\[
L=
\begin{pmatrix}
\frac{\partial f}{\partial t_1} & \frac{\partial f}{\partial t_2} & \ldots & \frac{\partial f}{\partial t_{n-1}}\\
\frac{\partial d}{\partial t_1} & \frac{\partial d}{\partial t_2} & \ldots & \frac{\partial d}{\partial t_{n-1}}\\
\end{pmatrix}.
\]
Finally we need to check for all real solutions which satisfy ${\bf u(t)} \geq 0$  which one minimizes the distance $d$. 

\begin{example}
Consider the matrix $M = \operatorname{diag}(1/4,1/25,1/36)$. Find a completion 
that minimizes the Euclidean distance from the uniform distribution:
\[
d=\sqrt{ \sum_{i,j}\left(m_{ij}-\frac{1}{n^2}\right)^2}=\sqrt{ \sum_{i,j}\left(u_iv_j-\frac{1}{n^2}\right)^2}.
\]
We use the Euclidean distance, but this method can be used for any distance measure. 

We construct the Lagrange matrix
\[
L=\begin{pmatrix}
\frac{\partial f}{\partial t_1} & \frac{\partial f}{\partial t_2}\\
\frac{\partial d}{\partial t_1} & \frac{\partial d}{\partial t_2}\\
\end{pmatrix} 
\]
and find the critical points of $d$ on the variety $f=1$ by solving the system
of rational equations $\{f=1,det(L)=0\}$. We use \texttt{maple} to construct $L$ and to solve the system of equations.
This system has $18$ solutions, out of which ten are real and four are feasible, i.e. they satisfy ${\bf u} \geq 0$.
The minimum is achieved at
\[
M=\begin{pmatrix}
0.250 & 0.049 & 0.215\\
0.204 & 0.040 & 0.176\\
0.032 & 0.006 & 0.028
\end{pmatrix}
\text{ and }
M^T.
\]
The Euclidean distance from the uniform distribution is $0.276$.
\end{example}

\subsection{Completions for Partial Matrices with Zeros} \label{subsection:boundary2}

Now we consider those partial matrices that have
zeros among their entries.

\begin{definition} Let $M$ be a partial matrix subordinate to
a pattern $S$.
Define $X(I,J; M)$ to be the set of completions of $M$
such that the rows in $I$ and the columns in $J$ are set
to zero. 
\end{definition}

This is only well-defined where $S \cap (I \cup J)$ are 
precisely the zero-labeled elements of $M$.
\begin{proposition}
Let $S$ be a pattern, and let $Z\subset S$ be the subset of zero entries.
Let $C$ be the set of vertex covers of $G(Z)$ so that no vertex
is adjacent to an edge of $G(S \setminus Z)$.
Then the set of completions for $M$ subordinate to $S$ is given
by:
\[ \bigcup_{c \in C} X(c; M).\]
\end{proposition}

\begin{proof}
All of the zero entries must be in a zero row or a zero
column - for this reason, we need a vertex cover. No nonzero
entry may be in a zero row or zero column. This is why no vertex
may be adjacent to an edge of $S \setminus Z$. The result follows
from there.
\end{proof}

\begin{figure}[h!]
\centering
\includegraphics[width=.5\textwidth]{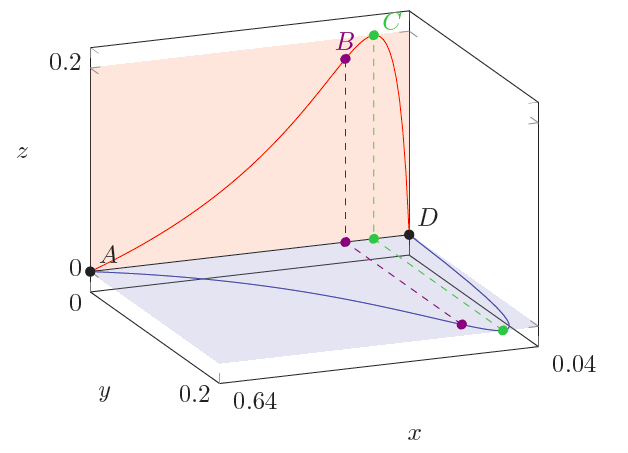}
\caption{Three coordinates of the completion set of $M$.}
\label{fig:comps}
\end{figure}

\begin{example}[Two components in solution set] 
Figure \ref{fig:comps} describes the set of completions for the matrix $M$,
by plotting all possible values of $x,y,z$; the indicated matrix
coordinates give an isomorphism between the two pictured curves
and all completions of $M$. Each curve comes from a minimal vertex
cover of the set of zero entries: the red curve takes column 3, and the
blue curve takes row 3. The components corresponding
to two minimal vertex covers intersect at points corresponding
to their union.
\[ M = \left[ \begin{array}{ccc}
.16 		& & \\
& .16	&  \\
&	& 0 
\end{array} \right]
\to \left[ \begin{array}{ccc}
.16 		& {\bf x}	& 6.25 xy \\
.0256 x^{-1}& .16	& {\bf y} \\
{\bf z}	& 6.25 xz	& 0 
\end{array} \right]\]

From left to right, we have the following matrices moving along the red curve:
\[ 
\begin{matrix}
A & B & C & D\\
\left[ \begin{array}{ccc}
.16 		& .64 & 0\\
.04 & .16	& 0 \\
0 &0	& 0 
\end{array} \right] &
\left[ \begin{array}{ccc}
.16 		& .16 & 0\\
.16 & .16	& 0 \\
.18 & .18	& 0 
\end{array} \right] &
\left[ \begin{array}{ccc}
.16 		& 8/75 & 0\\
6/25 & .16	& 0 \\
.2 & 2/15	& 0 
\end{array} \right] &
\left[ \begin{array}{ccc}
.16 		& .04 & 0\\
.64 & .16	& 0 \\
0 &0	& 0 
\end{array} \right] \\
\end{matrix}
\]

The matrix at point $B$ represents the point at which the
bottom row has largest possible row sum. The matrix at
point $C$ has the largest possible value for $z$.
Following along the blue curve, we would pass through
$B^T$ and $C^T$.
\end{example}

\section{Generalizations : Rank-1 Tensors and Rank-2 Matrices}
\label{section:generalizations}

In this section, we will consider the more general setting of tensors. 
The reader may consult~\cite{Landsberg} for an 
introduction on tensors and tensor rank. 
The notions of partial tensor, completion and pattern are analogous 
to the matrix case. Readers who are  interested only in matrices 
can read other sections independently of this section.

\subsection{Diagonal Partial Tensors}\label{subsection:tensors}

Theorem~\ref{halfnorm} for diagonal partial matrices generalizes nicely to diagonal partial tensors:

\begin{theorem} \label{thm:tensors}
Suppose we are given an order-$d$ $n \times n \times \cdots \times n$ partial tensor $T$ with nonnegative observed entries $a_i$ along the diagonal, i. e. we have $t_{ii\ldots i}=a_i$ for $1 \leq i \leq n$, and all other entries unobserved. Then $T$ is completable if and only if
\[ \sum_{i = 1}^n a_i^{1/d} \leq 1.\]
\end{theorem}

\begin{proof}
The proof is analogous to the proof of Theorem \ref{halfnorm}, with a few adjustments to deal with the multiple parametrizing vectors. The tensor $T$ can be factored as  ${\bf u}^1 \otimes \ldots \otimes {\bf u}^{d}$, where each ${\bf u}^i \in \Delta^{n-1} \subset \mathbb{R}^n$. The relations $a_i = u_i^1 \cdots u_i^d$ imply that the coordinates of ${\bf u}^d$ can be expressed as functions on the product of simplices $(\Delta^{n-1})^{d-1}$. Define $f:(\Delta^{n-1})^{d-1} \to \mathbb{R}$ by:
\[f({\bf u}^1,\ldots,{\bf u}^{d-1}) = \sum_{i=1}^n  u^d_i = \sum_{i = 1}^n \dfrac{a_i}{ u_i^1 \cdots u_i^{d-1}}.\]
Since every variable appears in the denominator of some term, 
the function $f$ approaches infinity at the boundary of the 
product of simplices. A candidate vector ${\bf u}^d$ will be 
available if and only if the minimum value of $f$ 
on the product of simplices is less than or equal to one.  
To find the minimum, compute partial derivatives; 
as in the proof of Theorem \ref{halfnorm}, we let $ u_n^j = 1 - \sum_{k = 1}^{n-1}  u_k^j$ 
for each $j = 1,\ldots, d-1$.
\[ \dfrac{\partial f}{\partial  u_i^j} = - \dfrac{a_i}{ u_i^j\prod_{k = 1}^{d-1}  u_i^k}  + \dfrac{a_n}{ u_n^j\prod_{k = 1}^{d-1}  u_n^k}.\]
Setting the partial derivatives to zero gives us the following:
\[  \dfrac{a_i}{ u_i^j\prod_{k = 1}^{d-1} u_i^k}  =  \dfrac{a_n}{ u_n^j\prod_{k = 1}^{d-1} u_n^k}  \hspace{1cm} \forall i \in [n-1], \forall j \in [d-1] .\]
Let $c_j$ be the value on both sides of this equation. Picking two values of $j$, w.l.o.g., $1$ and $2$, this designation means:
\[ \dfrac{a_i}{ u_i^1\prod_{k = 1}^{d-1}  u_i^k} /  \dfrac{a_i}{ u_i^2\prod_{k = 1}^{d-1} u_i^k} =  \dfrac{ u_i^2}{ u_i^1} = \dfrac{c_1}{c_2}.\]
Applying to all indices, we have $u_i^j = \frac{c_1}{c_j} u_i^1$ for all $i,j$. 
Since $\sum_i  u_i^j = 1 = \frac{c_1}{c_j} \sum_i  u_i^1 = c_1/c_j$ for all $j$, 
implying that every $c_j = c_1$, and
\[ \begin{matrix} \dfrac{a_i}{( u_i^j)^d}  = c_1 \hspace{5mm} \forall i \in [n], \forall j \in [d-1] & \Rightarrow   & u_i^j = (a_i/c_1)^{1/d} = \kappa(a_i)^{1/d}  . \end{matrix}\]
for some constant $\kappa$. Since the sum  $\sum_i  u_i^j = 1$, the value of $\kappa = (\sum_{i = 1}^n a_i^{1/d})^{-1}$.
Plugging in these values of $u_i^j$, we obtain 
\[ f = \sum_{i = 1}^n \dfrac{a_i}{(\kappa(a_i)^{1/d})^{d-1}} = \sum_{i = 1}^n \dfrac{ a_i^{1/d}}{\kappa^{d-1}} = \left(\sum_{i = 1}^n a_i^{1/d}\right)^d.\]
Since $f$ is at its minimum here, the value must be less than or equal to one for a solution to exist, proving the theorem.
\end{proof}

\subsection{General Partial Tensors}

Generalization to higher-order tensors brings several challenges. Theorem \ref{thm:tensors} gives a partial result characterizing diagonal tensors. However, the nice bipartite graph structure we had for matrices becomes $k$-partite hypergraphs with $k$-hyperedges; notions like connectivity will need to be modified. So, while any rank-one matrix completability problem was reducible to a diagonal case, the tensor case does not seem to be reducible in the same way. We record here the results for the smallest case distinct from matrices:

\begin{example}[$2 \times 2 \times 2$ Tensors]
The variety of $2 \times 2 \times 2$ rank-one tensors whose 
entries sum to one is $3$-dimensional. For this example,
we look only at algebraically independent sets of entries,
so as to extract the semialgebraic constraints.
We use the octahedral symmetry group of the cube 
to restrict to combinatorially distinct examples.
\begin{enumerate}
\item (Size $1$) Any singleton, e.g. $t_{000}$. The only condition is $t_{000} \leq 1$.
\item (Size $2$) Three orbits of pairs:
	\begin{enumerate}
	\item $t_{000},t_{001}$: $t_{000} + t_{001} \leq 1$.
	\item $t_{000},t_{011}$: $\sqrt{t_{000}} + \sqrt{t_{011}} \leq 1$.
	\item $t_{000},t_{111}$: $\sqrt[3]{t_{000}} + \sqrt[3]{t_{111}} \leq 1$.
	\end{enumerate}
\item (Size $3$) Three orbits of triples:
	\begin{enumerate}
	\item $t_{000},t_{001},t_{010}$: $t_{000} + t_{001} + t_{010} + (t_{001}t_{010}/t_{000}) \leq 1$.
	\item $t_{000},t_{001},t_{110}$: $\sqrt{t_{000}+t_{001}} + \sqrt{t_{110}+t_{001}t_{110}/t_{000}} \leq 1$.
	\item $t_{000},t_{101},t_{011}$: The tensor is completable if and only if the equation 
	\vspace{-6pt}
	\[ x^3 + (t_{000} + t_{101} + t_{011} - 1)x^2 + (t_{000} t_{101} + t_{000} t_{011} + t_{101} t_{011}) x + t_{000}t_{101}t_{011} = 0\] 
	
	\vspace{-8pt} has a root in the interval $[0,1]$.
	\end{enumerate}
\end{enumerate}
In this example, five of the seven cases are equivalent to
partial matrix problems. Only 2(c) and 3(c) use techniques that
do not fall immediately out of the matrix case; still, the
approach is analogous.
\end{example}

\subsection{Low-Rank Matrices}\label{subsection:conditional_independence}

One natural direction to generalize these results would be to fix rank $r> 1$, and find conditions for a matrix to be rank or nonnegative rank $r$ and have nonnegative entries that sum to one. One obvious consequence of our results is that any matrix completable to rank one is trivially completable as a higher-rank matrix. It is harder to provide tighter conditions, however, even in the smallest examples.
\begin{example}[$r = 2, m = n = 3$]\label{example:conditional_independence}
There are two polynomials constraining the entries of a $3 \times 3$ rank $2$ matrix in the standard simplex: the determinant must be zero, and the sum of the entries must be one. The variety of matrices with these properties has dimension $7$.

We take two combinatorially distinct partial matrices with seven entries. 
To find completions, we substitute $X$ and $R-X$ for the missing entries, 
where $R = 1 - (a + b + c + d + e + f +  g)$:
\[ \text{ A:   }
\left(
\begin{array}{ccc}
a & b & c \\
d & e & f \\
g & X  & R - X \\
\end{array}
\right)
\hspace{2cm}
\text{ B:   }
\left(
\begin{array}{ccc}
a & b & c \\
d & X & f \\
g & e & R - X\\
\end{array}
\right).
\]
Since the sum is now fixed at one, we only need to check that there is a value of $X$ in $[0,R]$ so that the determinant is zero. In the first case the determinant gives a linear equation in $X$, while in the second case the determinant is a quadratic; the solutions to each are:
\begin{align}
X &= \frac{g(bf - ce)+ R(ae - bd)}{(ae - bd) + (af - cd)} \nonumber\\
X &= \frac{(aR + bd - cg) \pm \sqrt{ (aR + bd - cg)^2 - 4a(b(dR - fg) + e(af - cd))}}{2a} \nonumber
\end{align}
Substituting the values $(.07, .09,.09,.12,.15,.04,.16)$ 
in each matrix yields a completion 
for A, but since the discriminant of B is negative, 
no completion is possible.
\end{example}

From the statistics viewpoint, it would be more interesting to study completability to matrices in the standard simplex of nonnegative rank at most $r$, because the $r$-th mixture model of two discrete random variables is the semialgebraic set of matrices of nonnegative rank at most $r$. If a nonnegative matrix has rank $0,1,$ or $2$, then its nonnegative rank is equal to its rank. Hence, in Example~\ref{example:conditional_independence}, we simultaneously address the question of completing a partial matrix to a matrix in the standard simplex of nonnegative rank $2$. 

If $r\geq 3$, then matrices of nonnegative rank at most $r$ form a complicated semialgebraic set. For $r=3$, a semialgebraic description of this set is given in~\cite[Theorem~3.1]{KubjasRobevaSturmfels}. Partial matrices that are completable to matrices in the standard simplex of nonnegative rank at most $3$ are coordinate projections of this semialgebraic set.

\section{Semialgebraic Description}\label{subsection:semialgebraic_set}

A reader interested only in matrices can replace everywhere in this section ``tensor" by ``matrix".

\begin{proposition}\label{prop:completable_matrices_form_a_semialgebraic_set}
Partial tensors subordinate to a pattern $S$ that are completable to rank-one  tensors in the standard simplex form a semialgebraic set.
\end{proposition}

\begin{proof}
The independence model is a semialgebraic set defined by $2 \times 2$-minors of all flattenings, nonnegativity constraints and entries summing to one. The statement of the proposition follows by the Tarski-Seidenberg theorem.
\end{proof}

The goal of this section is to find a semialgebraic description of this semialgebraic set, see~\cite{RealAlgebraicGeometry} for an introduction to real algebraic geometry.  The difference from characterizations in Theorems~\ref{halfnorm} and~\ref{main_theorem} is that we aim to derive a description without square roots.  
For $2 \times 2$ partial matrices with diagonal entries, a semialgebraic description is given in Example~\ref{2by2} and its derivation from the inequality containing square roots is explained in Remark~\ref{remark:2x2_case_example_vs_theorem}.

We will characterize the semialgebraic set of diagonal partial tensors which can be completed to rank-one tensors in the standard simplex. This is the positive part of the unit ball in the $L^{\frac{1}{d}}$ space.

\begin{proposition}\label{prop:semialgebraic_description}
There exists a unique irreducible polynomial $f$ of degree $d^{n-1}$ with 
constant term one that vanishes on the boundary of the set of diagonal partial tensors $T \subset (\mathbb{R}^n)^{\otimes d}$ which can be completed to rank-one
 tensors in the standard simplex.
The semialgebraic description takes the form $f \geq 0$, coordinates $\geq 0$
plus additional inequalities that separate our set from other chambers
in the region defined by $f \geq 0$.
\end{proposition}

The proof of Proposition~\ref{prop:semialgebraic_description} was suggested to us by Bernd Sturmfels. For analogous proof idea, see~\cite[Lemma 2.1]{NieParriloSturmfels}.

\begin{proof}
Denote the diagonal entries of the partial tensor by $x_1, \ldots ,x_n$. We will show that the defining polynomial of the $\frac{1}{d}$-unit ball can be written as
\begin{equation}\label{equation:unit_ball} 
p_{d,n}= \prod _{\substack{(y_1, \ldots, y_{n-1}) \text{ s.t. } y_i^d = x_i  \\ \text{ for each } i}} \left((1-y_1-...-y_{n-1})^d-x_n\right).
\end{equation}
We want to eliminate $y_1,\ldots ,y_n$ from the ideal
\[
I=\left\langle y_1^d-x_1,\ldots, y_n^d-x_n,\sum_{i=1}^n y_i-1 \right\rangle \subset \mathbb{Q}[x_1,\ldots ,x_n,y_1,\ldots ,y_n].
\]
First replace $y_n$ by $1-y_1-...-y_{n-1}$ in the equation  $y_n^d-x_n$.
We consider
the field of rational functions $K = \mathbb{Q}(x_1,\ldots ,x_n)$. Solving the first $n-1$ equations $y_i^d-x_i$ 
is equivalent to adjoining the $d$-th roots of
$x_i$ for $i \in \{1,\ldots ,n-1\}$ to the base field.
This gives an extension $L$ of degree $d^{n-1}$ over $K$. 
The group $\text{Aut}_K(L)$ of all automorphisms of $L$ that leave $K$ fixed is $(\mathbb{Z}/d\mathbb{Z})^{n-1}$.
The product over all elements $(1-y_1-...-y_{n-1})^d-x_n$ in the orbit of $\text{Aut}_K(L)$ gives~(\ref{equation:unit_ball}), and thus lies in the base field $K$.
Every factor in the product~(\ref{equation:unit_ball}) is integral
over $\mathbb{Q}[x_1,\ldots ,x_n]$, hence the product~(\ref{equation:unit_ball})
is a degree $d^{n-1}$ polynomial in $x_1,\ldots ,x_n$.
No subproduct is left invariant under the automorphism group, so~(\ref{equation:unit_ball})
is irreducible.
\end{proof}

\bigskip \bigskip

\paragraph{{\bf}ACKNOWLEDGEMENTS}
We thank Bernd Sturmfels for introducing this problem to us, 
suggesting the proof of Proposition~\ref{prop:semialgebraic_description}
 and providing detailed feedback on the first draft of this article; 
Vishesh Karwa and Aleksandra Slavkovi\'c for suggesting the problem; 
Thomas Kahle for sharing his ideas on the project;  
Mario Kummer for correcting a mistake in the proof of 
Proposition~\ref{prop:semialgebraic_description}; 
Louis Theran for helping us with the complexity of algorithms.
We also thank the anonymous referees for correcting several mistakes
and aiding in the exposition. 
This collaboration was initiated while both authors were 
guests of the Max-Planck Institute for Mathematics, 
and was continued at the as2014 conference at Illinois 
Institute of Technology.

\bibliography{biblio}
\bibliographystyle{plain}
\end{document}